\documentclass{article}
\usepackage[T1]{fontenc}
\usepackage[english]{babel}
\usepackage{amsmath}
\usepackage{amsthm,amssymb}
\usepackage{authblk}
\usepackage{amssymb}
\usepackage{xspace}
\usepackage{hyperref}
\usepackage{verbatim}
\usepackage{tikz}
\usepackage[procnumbered,linesnumbered,ruled,vlined]{algorithm2e}
\usepackage{tabularray}
\usepackage{cleveref}
\usepackage[textsize=footnotesize,color=green!40]{todonotes}
\usepackage{stmaryrd}
\usepackage{comment}
\usepackage{subcaption}
\usepackage{apxproof}
\usepackage[left= 3 cm,right=3 cm,top=3 cm,bottom=3 cm]{geometry}

\usetikzlibrary{shapes.misc}
\usetikzlibrary{calc}
\tikzstyle{noeud}=[circle,inner sep=2, minimum size =3 pt, line width = 1pt, draw=black, fill=white]

\newtheorem{theorem}{Theorem}

\newtheorem{proposition}[theorem]{Proposition}
\newtheorem{corollary}[theorem]{Corollary}
\newtheorem{lemma}[theorem]{Lemma}

\title{On four network monitoring parameters in graphs and their gaps}
\author[1,2]{Zin Mar Myint}
\author[1]{Avikal Srivastava}
\affil[1]{Indian Institute of Technology Dharwad, India}
\affil[2]{Polytechnic University (Kyaing Tong), Myanmar}

\date{}

\begin{document}


\maketitle

\begin{abstract}
Let \( G \) be a finite simple undirected graph. Four graph parameters related to network monitoring are the \emph{geodetic set}, \emph{edge geodetic set}, \emph{strong edge geodetic set}, and \emph{monitoring edge geodetic set}, with corresponding minimum sizes, denoted by \( g(G), eg(G), seg(G) \), and \( meg(G) \), respectively. These parameters quantify the minimum number of vertices required to monitor all vertices and edges of \( G \) under progressively stricter path-based conditions.  As established by Florent \textit{et al.}\ (CALDAM 2023), these parameters satisfy the chain of inequalities:
\(
g(G) \leq eg(G) \leq seg(G) \leq meg(G).
\) In 2025, Florent \textit{et al.}\ posed the following question: given integers \( a, b, c, d \) satisfying \( 2 \leq a \leq b \leq c \leq d \), does there exist a graph \( G \) such that
\(
g(G) = a, \quad eg(G) = b, \quad seg(G) = c, \quad \text{and} \quad meg(G) = d?
\)
They partially answered this affirmatively under three specific hypotheses and gave some constructions to support it. In this article, we first identify quadruples of values that cannot be realized by any connected graph. For all remaining admissible quadruples, we provide explicit constructions of connected graphs that realize the specified parameters. These constructions are modular and efficient, with the number of vertices and edges growing linearly with the largest parameter, providing a complete and constructive characterization of such realizable quadruples.
\end{abstract}


\section{Introduction}
Distance-based monitoring through graph theoretic models provide a fundamental framework for network path observability, fault detection and routing verification in network problems.
A key minimization problem in this area is to determine the minimal subset of strategically selected vertices such that they reveal the structure of shortest paths in the graph.
To capture different levels of monitoring or fault-detection capability, several vertex-based parameters have been introduced, each imposing progressively stronger requirements on how vertices or edges are covered by a selected collection of "probes" whose shortest path interactions expose structural information or detect faults in the network.

There are four parameters relevant to the context of this article, namely, the \emph{geodetic number} \( g(G) \), the \emph{edge-geodetic number} \( eg(G) \), the \emph{strong edge-geodetic number} \( seg(G) \), and the \emph{monitoring edge-geodetic number} \( meg(G) \). They play central roles in measuring the efficiency of probe placement in communication networks.

A set \(S \subseteq V(G)\) is called a \emph{geodetic set} if every vertex \(x \in V(G)\) lies on any shortest path between some pair of distinct vertices in \(S\).
The minimum cardinality of a geodetic set is the \emph{geodetic number} of \(G\), denoted by \(g(G)\). 
This concept was first introduced by Harary \emph{et al.}~\cite{Harary1993}.

An \emph{edge-geodetic set} of a graph \(G\) is a vertex subset \(S \subseteq V(G)\) such that
every edge of \(G\) lies on any shortest path between some pair of distinct vertices in \(S\).
The minimum cardinality of an edge-geodetic set is the \emph{edge-geodetic number} of \(G\),
denoted by \(eg(G)\), a notion introduced by Atici \emph{et al.}~\cite{Atici2003}.

Later, Manuel \emph{et al.}~\cite{Manuel2017} refined this notion further by defining the strong edge-geodetic set, which states that, instead of allowing all shortest
paths between each pair of selected vertices, one must designate a \emph{specific} shortest
path for each pair, and the union of these chosen paths must cover all edges of the graph.
Hence, a \emph{strong edge-geodetic set} of the graph \(G\) is a vertex subset \(S \subseteq V(G)\) together with a choice
of a specific shortest path \(P_{uv}\) for each unordered pair of distinct vertices \(u,v \in S\),
such that every edge of \(G\) lies on at least one of these chosen paths. The minimum cardinality of such a set is the \emph{strong edge-geodetic number}, denoted \(seg(G)\).

The most recent and most restrictive parameter in this family is the
\emph{monitoring edge-geodetic set} (MEG-set), introduced by Foucaud
\emph{et al.}~\cite{foucaud2023monitoring} in the context of fault detection.  
Here, the requirement is that every edge must be monitored by some pair of vertices in $S\subseteq V(G)$ where
a pair $u,v \in S$ \emph{monitors} an edge $e$ if $e$ lies on every shortest $u$--$v$ path in $G$. The minimum cardinality of such a set is the \emph{monitoring edge-geodetic number} of \(G\),
denoted \(meg(G)\).

These four parameters are known ~\cite{foucaud2025} to hold the following relation \[
g(G) \leq eg(G) \leq seg(G) \leq meg(G),
\]
and each corresponds to increasingly stringent monitoring conditions capturing a progressively stronger notion of coverage. Moreover, the relationships among these four \( g(G), eg(G), seg(G) \), and \( meg(G) \) were analyzed in detail in
Foucaud \emph{et al.}~\cite{foucaud2025}, where the authors initiated a systematic study of the
possible differences between them. 
A central problem posed in that work is the following:
\begin{quote}
    In \cite{foucaud2025}, given integers $2 \le a \le b \le c \le d$, does there exist a connected graph $G$ such that
    \(g(G)=a,\ eg(G)=b,\ seg(G)=c,\ meg(G)=d \, ?\) 
\end{quote}
In the same paper, the authors constructed graphs showing that nearly all quadruples 
\( 4 \leq a \leq b \leq c \leq d \) are realizable, subject to a small number of exceptional configurations. Their constructions rely on bipartite base graphs with pendent attachments and carefully arranged subdivisions, using interactions among simplicial vertices, creating the clique and twin structures to control the parameter values. However, some specific cases, particularly those involving smaller values or tight equalities between consecutive parameters (for instance, when \( d = c + 1 \) with \( b = a + 1 \)), were excluded by the constraints of the original construction. Moreover, the structural behavior of the resulting graphs suggests that further refinement and generalization of the construction might bridge these exceptional gaps and lead to a complete characterization for all positive integers \( 2 \leq a \leq b \leq c \leq d \).

\subsection{Our contributions}
In this article, we resolve the realizability problem completely.  
We introduce a modified and structurally flexible construction that produces, for \emph{every} integer quadruple \(
    2 \le a \le b \le c \le d, \)
a connected graph $G$ satisfying
$g(G)=a$, $eg(G)=b$, $seg(G)=c$, and $meg(G)=d$.
Our approach refines the building of construction that was used in ~\cite{foucaud2025}, allowing for precise
control over the roles of simplicial vertices, twins, pendant attachments, and subdivided
paths, and eliminating all previously unresolved cases.  
This establishes the first complete description of the achievable combinations of the four
monitoring-related parameters.

\begin{enumerate}
    \item [(i)] Section~\ref{sec:prelim} recalls the corresponding definitions and known results that will be used throughout the article, 
including structural lemmas for vertices that must belong to every monitoring edge-geodetic set.
\item [(ii)] Section~\ref{main_results} is divided into four parts:
\begin{enumerate}
    \item \Cref{eliminating} shows that there does not exist any connected graph $G_{a,b,c,d}$ for some specific parameters,
    \item \Cref{subsec:23-realization} provides modified constructions which prove the realizability of $g(G)=2$ and $eg(G)=3$. 
    
\item \Cref{subsec:g2-general} extends the general realization for $g(G)=2$ and $eg(G)\ge 4$, 
and 
\item \Cref{subsec:general-construction} generalizes the constructions to arbitrary parameter quadruples 
and analyze the computational complexity of our construction, showing that 
    $G_{a,b,c,d}$ can be built in $O(d)$ time, with the number of vertices and edges growing linearly with the largest parameter $d$.
\end{enumerate}

\item [(iii)] \Cref{sec:conclusion} concludes the article with open directions for future research.    
\end{enumerate}




\section{Preliminaries and known results}
\label{sec:prelim}

Throughout this article, all graphs are finite, simple, and connected.
For a graph $G$, we denote its vertex set by $V(G)$ and its edge set by $E(G)$.

A vertex \(v \in V(G)\) is \emph{simplicial} if its neighborhood induces a clique in \(G\), and it is
called \emph{pendent} if \(\deg_G(v)=1\).  
A vertex \(v\) is a \emph{cut vertex} if the graph \(G \setminus \{v\}\) is disconnected.  
A \emph{clique} of \(G\) is an induced subgraph that is a complete graph.

\begin{lemma}[\cite{foucaud2023monitoring}]\label{cut vertex}
    Let $G$ be a graph with a cut-vertex $u$. Then $u$ is never part of any minimal MEG-set of $G$.
\end{lemma}

\begin{lemma}[\cite{foucaud2023monitoring}]\label{simplicial2023}
In a connected graph $G$ with at least one edge, any simplicial vertex belongs to any edge-geodetic set and thus, to any MEG-set of $G$.
\end{lemma}

It was shown in~\cite{foucaud2023monitoring} that this implies that all pendant vertices of a graph \(G\) will be part of any geodetic set, edge-geodetic set, strong edge-geodetic
set, and monitoring edge-geodetic set by \cref{simplicial2023}. 

Two distinct vertices $u,v$ are \emph{open twins} if $N(u)=N(v)$, and \emph{closed twins} if
$N[u]=N[v]$.  
In this work, we simply refer to them as \emph{twins}. Twins behave rigidly with respect to monitoring.

\begin{lemma}[{\cite{foucaud2023monitoring}}]
\label{lem:twins}
Every pair of (open or closed) twins of degree at least one belongs to every MEG-set.
\end{lemma}

An important structural link between MEG-sets and strong edge-geodetic sets was established
in~\cite{foucaud2025}.

\begin{proposition}[\cite{foucaud2025}]\label{segiffmeg}
    Let $S \subseteq V(G)$ be a vertex subset of a graph $G$ and let $f$ be an assignment of a shortest path to each pair of
vertices of $S$. Then $S$ is an MEG-set if and only if $S$, along with the assignment $f$, is a strong edge-geodetic set for any choice of $f$.
\end{proposition}

\begin{theorem}[\cite{foucaud2025}]\label{4-cycle}
Let $G$ be a graph. A vertex $v \in V(G)$ is in every MEG-set of $G$ if and only if there exists $u \in N(v)$ such that for any vertex $x \in N(v)$, any induced 2-path $uvx$ is part of a 4-cycle.
\end{theorem}

Building on these structural observations, the next section first determines 
which parameter quadruples can occur in connected graphs, and then develops a 
refined construction that overcomes all remaining limitations.  
In particular, we show that every quadruple 
$2 \le a \le b \le c \le d$ is realizable.



\section{Main Results}\label{main_results}

The structural observations collected in Section~\ref{sec:prelim} impose 
strong restrictions on how the four parameters 
\(g(G)\), \(eg(G)\), \(seg(G)\), and \(meg(G)\) may interact.  
Before presenting our constructions, we begin to use these constraints to eliminate several parameter quadruples 
that any connected graph cannot realize.  
These results form the first step toward a full characterization.

\subsection{Eliminating impossible parameter quadruples}\label{eliminating}
We first identify certain parameter combinations that cannot be realized, 
clarifying structural limitations imposed by the preliminary results.

\begin{theorem}    \label{thm:no_222d}
For any integer \(d>2\), there does not exist a connected graph \(G\) satisfying
\[
g(G) = eg(G) = seg(G) = 2 \quad \text{and} \quad meg(G) = d.
\]
\end{theorem}

\begin{proof}
Assume, for the sake of contradiction, that such a connected graph \(G\) exists.  
Since \(g(G) = 2\), there must exist a pair of vertices \(u,v\) such that every vertex of \(G\) lies on some shortest \(u\text{--}v\) path.

Now, \(eg(G)=2\) implies that every edge of \(G\) also lies on a shortest path between \(u\) and \(v\).

Since \(seg(G)=2\), there exists a set \(S=\{u,v\}\subseteq V(G)\) and an assignment that selects a particular shortest \(u\!-\!v\) path \(P_{uv}\) such that every edge of \(G\) lies on the chosen path \(P_{uv}\). In other words, the set of edges of \(G\) is contained in the edge-set of \(P_{uv}\); hence every edge of \(G\) is an edge of \(P_{uv}\).

Consequently, there cannot exist any shortest \(u\!-\!v\) path different from \(P_{uv}\): any alternative shortest path would use an edge not in \(P_{uv}\), contradicting that all edges of \(G\) lie on \(P_{uv}\). Thus \(P_{uv}\) is the unique shortest \(u\!-\!v\) path in \(G\). As there is a unique shortest $u$-$v$ path in $G$ and this path covers all the vertices and edges of $G$, the graph $G$ is precisely $P_{uv}$. Hence, $meg(G)=meg(P_{uv})=2$, contradicting the assumption $meg(G)=d>2$.

This contradiction shows that no such graph \(G\) exists. 
\end{proof}
We now exclude another set of parameter quadruples that are incompatible, specifically those with \(eg(G) = seg(G) = 3\).

\begin{theorem}\label{thm:no_233d}
There does not exist a connected graph $G$ with parameters
\[
g(G)=2, \quad eg(G)=3, \quad seg(G)=3, \quad meg(G)=d \ge 3.
\]
In other words, $G_{(2,3,3,d)}$ graphs do not exist.
\end{theorem}

\begin{proof}
Since $g(G)=2$, let $\{u,v\}$ be a geodetic set of $G$.  
Then every vertex of $G$ lies on some shortest $u$--$v$ path.

However, since $eg(G)=3$, the pair $\{u,v\}$ does not cover all edges.  
Therefore, there exists at least one edge $xy\in E(G)$ which does not lie on any shortest 
$u$--$v$ path.   
We observe that each of $x$ and $y$ appears on at least one shortest $u$--$v$ path, 
but the edge $xy$ never appears on any such path.
Since $eg(G)=3$ and $xy$ is not contained in any shortest $u$--$v$ path, both $x$ and $y$ must be included in the edge-geodetic set to 
cover the edge $xy$. 

Due to this argument, $x$ and $y$ must be included in the strong edge-geodetic set as well.
Since $seg(G)=3$, let us say $S=\{x,y,w\}$ which covers 
all edges.
By definition of the strong edge-geodetic number, there exist specific shortest paths
\(
P_{xy}, \quad P_{xw}, \quad P_{yw}
\)
whose union covers all edges of $G$. This structure forces the union 
\[
P_{xw} \cup P_{yw} \cup \{xy\}
\]
to contain at least one cycle. Note that there is no pendent vertex because every pandent vertex belongs to the intersection of all $4$ parameters. Hence, $u,v,x,y,w$ are not pendent vertices.

In particular, the edge $xy$ lies on $P_{xy}$, 
but $P_{xy}$ does not intersect any shortest $u$--$v$ path. At the same time, $P_{xw}$ and $P_{yw}$ must cover all remaining edges and pass through all vertices of the graph $G$. 
Due to the definition of a strong edge-geodetic set, there are unique shortest paths between $x$--$w$ and $y$--$w$. Without loss of generality, let there be two shortest paths between $x$ and $w$, then $S$ cannot cover all edges of the graph $G$. Hence, there is a unique shortest path between $x$--$w$ and $y$--$w$. Thus, $G$ is the form as $xy\cdots w\cdots x$ and it is a cycle in which every edge of the cycle must lie on some shortest $u$--$v$ path.  
But the cycle contains the edge $xy$, which by assumption lies on \emph{no} shortest 
$u$--$v$ path. Therefore, it is a contradiction that no matter whether even or odd cycle, since $g(G)\neq eg(G)$.

This is impossible. Therefore, no such graph $G$ exists. 
\end{proof}

Having identified the parameter quadruples with $g(G)=2$ and $eg(G)=3$ that cannot be realized,
we now turn to the complementary task of establishing existence results.

\subsection{Realizing the case \texorpdfstring{$g(G)=2$ and $eg(G)=3$}{g(G)=2 and eg(G)=3}}\label{subsec:23-realization}

In this subsection, we construct explicit families of connected graphs
realizing all remaining admissible quadruples with
\[
g(G)=2 \quad \text{and} \quad eg(G)=3,
\]
thereby completing the analysis for this parameter $2<3< c\leq d$.

\begin{lemma}\label{lem:eg-implies-seg}
Let $G$ be a connected graph and let $u \in V(G)$.
If $u$ belongs to every edge-geodetic set of $G$, then
$u$ belongs to every strong edge-geodetic set of $G$.
\end{lemma}
\begin{proof}
Let $u$ belongs to every edge-geodetic set of $G$.
Suppose, for contradiction, that there exists a strong edge-geodetic set
$S \subseteq V(G)$ such that $u \notin S$.
By the definition of a strong edge-geodetic set, for every unordered pair
$\{x,y\} \subseteq S$ one specific shortest $x$--$y$ path $P_{xy}$ is chosen and the union of these paths covers all edges of $G$. Since every edge of $G$ lies on \emph{some} shortest path between
a pair of vertices of $S$.
Hence, $S$ is also an edge-geodetic set of $G$.

This contradicts that $u$ belongs to every edge-geodetic set.
Therefore, $u$ must belong to every strong edge-geodetic set of $G$.
\end{proof}

\begin{lemma}\label{lem:eg-implies-meg}
Let $G$ be a connected graph and let $u \in V(G)$.
If $u$ belongs to every strong edge-geodetic set of $G$, then
$u$ belongs to every monitoring edge-geodetic set of $G$.
\end{lemma}

\begin{proof}
Let $u$ belong to every strong edge-geodetic set of $G$.
Suppose, for the sake of contradiction, that there exists a monitoring edge-geodetic set
$S \subseteq V(G)$ such that $u \notin S$.

Since $S$ is a monitoring edge-geodetic set, $u$ lies on all shortest paths between some pair of vertices in $S$. Let $f$ be an assignment of a shortest path to each pair of vertices of $S$. Then, by the
\Cref{segiffmeg} that $S$, along with the assignment $f$, is a strong edge-geodetic set for any choice of $f$. 
Thus, there exists a strong edge-geodetic set of $G$ that does not contain $u$,
contradicting the assumption that $u$ belongs to every strong edge-geodetic set
of $G$.
Therefore, $u$ must belong to every monitoring edge-geodetic set of $G$.
\end{proof}

\begin{corollary}\label{eg-seg-meg}
Let $G$ be a connected graph.
Any vertex contained in every edge-geodetic set of $G$ is contained in every
monitoring edge-geodetic set of $G$.
\end{corollary}

\begin{proof}
The result follows immediately from
Lemmas~\ref{lem:eg-implies-seg} and~\ref{lem:eg-implies-meg}.
\end{proof}

We now prove that for every $c\ge 4$ and every $d\ge c$, there exists
a connected graph realizing the parameter quadruple
$(g,eg,seg,meg)=(2,3,c,d)$.

\medskip

\begin{theorem}\label{23cd}
For any integer $c \ge 4$, there exists a connected graph $G_{2,3,c,d}$ such that
\[
g(G_{2,3,c,d}) = 2, \quad eg(G_{2,3,c,d}) = 3, \quad
seg(G_{2,3,c,d}) = c, \quad meg(G_{2,3,c,d}) = d
\].
\end{theorem}

\begin{proof}
\begin{figure}[ht]
\centering
\begin{tikzpicture}[
  scale=0.75,
  vertex/.style={circle,draw,fill=black,inner sep=1pt,minimum size=1.8mm},
  small/.style={circle,draw,fill=black,inner sep=1pt,minimum size=1.8mm},
  every edge/.style={thick}
  ]

\node[vertex,label=left:$x_0$] (x0) at (-3,1) {};
\node[vertex,label=above left:$w_0$] (w0) at (-2,1) {};
\node[vertex,label=above:$u_0$] (u0) at (-1,1) {};
\node[vertex,label=above:$y_0$] (u1) at (0,1) {};
\foreach \i in {0,...,4} {
\node [vertex,label=above:$v_{0\i}$] (v0\i) at (\i+1, 1) {};
    }   
\node [vertex,label=above left:$v_{0{(r-2)}}$] (v0r-2) at (7, 1) {};
\node [vertex,label=above :$v_{0{(r-1)}}$] (v0r-1) at (8, 1) {};
\node [vertex,label=right :$v_{0r}$] (v0r) at (9,1) {};

\draw (x0)--(w0)--(u0)--(u1)--(v00)--(v01)--(v02)--(v03)--(v04);
\draw [thick, dotted] (v04)--(v0r-2);
\draw (v0r-2)--(v0r-1)--(v0r);

\node[vertex,label=left:$x_1$] (x1) at (-3,-1) {};
\node[vertex,label=above left:$w_1$] (w1) at (-2,-1) {};
\node[vertex,label=above:$u_1$] (u11) at (-1,-1) {};
\node[vertex,label=above:$y_1$] (u'1) at (0,-1) {};
\foreach \i in {0,...,4} {
\node [vertex,label=above:$v_{1\i}$] (v1\i) at (\i+1, -1) {};
    }   
\node [vertex,label=above left:$v_{1{(r-2)}}$] (v1r-2) at (7, -1) {};
\node [vertex,label=below :$v_{1{(r-1)}}$] (v1r-1) at (8, -1) {};
\node [vertex,label=right :$v_{1r}$] (v1r) at (9,-1) {};

\draw (x1)--(w1)--(u11)--(u'1)--(v10)--(v11)--(v12)--(v13)--(v14);
\draw [thick, dotted] (v14)--(v1r-2);
\draw (v1r-2)--(v1r-1)--(v1r);
  
\foreach \j in {2,3}  {
\node[vertex,label=left:$x_\j$] (x\j) at (-3,-\j) {};
\node[vertex,label=above:$w_\j$] (w\j) at (-2,-\j) {};
\node[vertex,label=above:$u_\j$] (u\j) at (-1,-\j) {};
\node[vertex,label=above:$y_\j$] (u2\j) at (0,-\j) {};

\foreach \i in {0,...,4} {
\node [vertex,label=above:$v_{\j\i}$] (v\j\i) at (\i+1, -\j) {};
    } 
\node [vertex,label=above :$v_{\j{(r-2)}}$] (v\j r-2) at (7, -\j) {};
\node [vertex,label=below :$v_{\j{(r-1)}}$] (v\j r-1) at (8, -\j) {};
\node [vertex,label=right :$v_{\j r}$] (v\j r) at (9,-\j) {};

\draw (x\j)--(w\j)--(v\j0)--(v\j1)--(v\j2)--(v\j3)--(v\j4);
\draw [thick, dotted] (v\j4)--(v\j r-2);
\draw (v\j r-2)--(v\j r-1)--(v\j r);
}

\node[vertex,label=above:$z$] (z) at (-1.5,2) {};
\draw (w0)--(z);
\draw (w1)--(z);
\draw (x0)--(x1)--(x2)--(x3);
\draw (v0r)--(v1r)--(v2r)--(v3r);
\draw (z)--(u0);

\node[vertex,red,label=above:$v'_{1_{r-1}}$] (v'1r-1) at (8,0) {};
\draw[dotted, thick, red] (v1r-2)--(v'1r-1)--(v1r);

\node [vertex,label=above:$v'_{01}$] (v'01) at (2, 2) {};
\node [vertex,label=above:$v'_{03}$] (v'03) at (4, 2) {};
\node [vertex,label=above:$v'_{0{(r-1)}}$] (v'0r-1) at (8, 2) {};

\draw (v00)--(v'01)--(v02)--(v'03)--(v04);
\draw (v0r-2)--(v'0r-1)--(v0r);

\node[vertex,label=left:$x_{c-4}$] (xc-4) at (-3,-5) {};
\node[vertex,label=above :$w_{c-4}$] (wc-4) at (-2,-5) {};
\node[vertex,label=above:$u_{c-4}$] (uc-4) at (-1,-5) {};
\node[vertex,label=above:$y_{c-4}$] (u'c-4) at (0,-5) {};
\node [vertex,label=above:$v_{{(c-4)}0}$] (vc-40) at (1, -5) {};\node [vertex,label=below:$v_{{(c-4)}1}$] (vc-41) at (2, -5) {};\node [vertex,label=above:$v_{{(c-4)}2}$] (vc-42) at (3, -5) {};\node [vertex,label=below:$v_{{(c-4)}3}$] (vc-43) at (4, -5) {};\node [vertex,label=above:$v_{{(c-4)}4}$] (vc-44) at (5, -5) {};
 
\node [vertex,label=above:$v_{{(c-4)}{(r-2)}}$] (vc-4r-2) at (7, -5) {};
\node [vertex,label=below :$v_{{(c-4)}{(r-1)}}$] (vc-4r-1) at (8, -5) {};
\node [vertex,label=right :$v_{{(c-4)}r}$] (vc-4r) at (9,-5) {};

\draw (xc-4)--(wc-4)--(uc-4)--(u'c-4)--(vc-40)--(vc-41)--(vc-42)--(vc-43)--(vc-44);
\draw [thick, dotted] (vc-44)--(vc-4r-2);
\draw (vc-4r-2)--(vc-4r-1)--(vc-4r);

\node[vertex,label=left:$x_{c-3}$] (xc-3) at (-3,-6.5) {};
\node[vertex,label=above:$w_{c-3}$] (wc-3) at (-2,-6.5) {};
\node[vertex,label=above:$u_{c-3}$] (uc-3) at (-1,-6.5) {};
\node[vertex,label=above:$u'_{c-3}$] (u'c-3) at (0,-6.5) {};

\node [vertex,label=below:$v_{{(c-3)}0}$] (vc-30) at (1, -6.5) {};\node [vertex,label=above:$v_{{(c-3)}1}$] (vc-31) at (2, -6.5) {};\node [vertex,label=below:$v_{{(c-3)}2}$] (vc-32) at (3, -6.5) {};\node [vertex,label=above:$v_{{(c-3)}3}$] (vc-33) at (4, -6.5) {};\node [vertex,label=below:$v_{{(c-3)}4}$] (vc-34) at (5, -6.5) {};
  
\node [vertex,label=above :$v_{{(c-3)}{(r-2)}}$] (vc-3r-2) at (7, -6.5) {};
\node [vertex,label=below :$v_{{(c-3)}{(r-1)}}$] (vc-3r-1) at (8, -6.5) {};
\node [vertex,label=right :$v_{{(c-3)}r}$] (vc-3r) at (9,-6.5) {};

\draw (xc-3)--(wc-3)--(uc-3)--(u'c-3)--(vc-30)--(vc-31)--(vc-32)--(vc-33)--(vc-34);
\draw [thick, dotted] (vc-34)--(vc-3r-2);
\draw (vc-3r-2)--(vc-3r-1)--(vc-3r);

\draw[dotted, thick] (x3)--(xc-4);
\draw[dotted, thick] (v3r)--(vc-4r);

\draw (xc-4)--(xc-3);
\draw (vc-4r)--(vc-3r);

\end{tikzpicture}
\caption{The graph $G_{2,3,c,d}$ for $c\ge 4$. 
(The vertex $v'_{1(r-1)}$ and the dotted edges are added when $d-c \ge 2$ is even.)}
\label{$G_{2,3,c,d}$}
\end{figure}
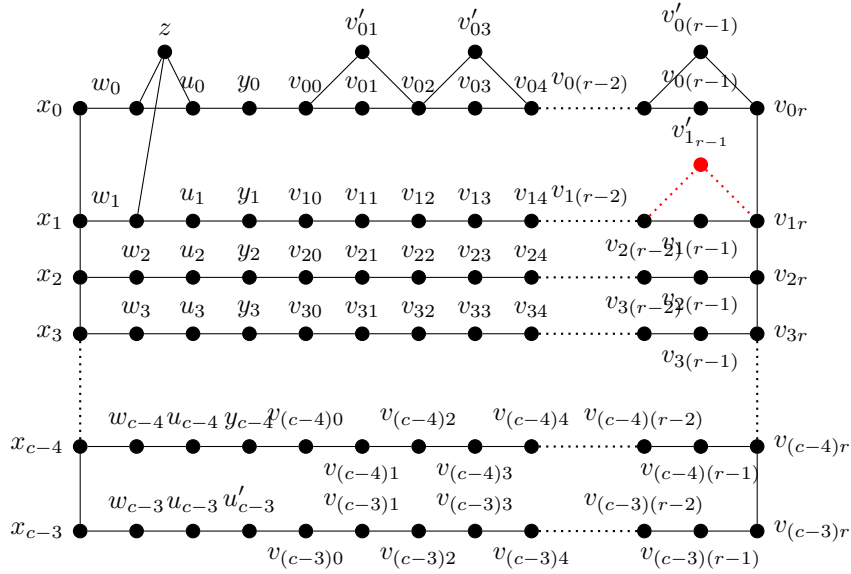

The graph $G_{2,3,c,d}$ is constructed as shown in \Cref{$G_{2,3,c,d}$}.
We define the variable $r$ as follows:
\[
r \;=\; \begin{cases}
d-c+1, &\text{if } d-c \text{ is odd},\\[1mm]
d-c, &\text{if } d-c \text{ is even},\\[1mm]
\end{cases}
\]
It consists of $c-2$ parallel paths of length $r+4$ labelled
$x_i w_i u_i u'_i v_{i_0} v_{i_1} \cdots v_{i_r}$ for $i \in \{0,\dots,c-3\}$, together with edges
$x_i x_{i+1}$ and $v_{ir} v_{{(i+1)}r}$ for $0 \le i \le c-4$.
We add twins $v'_{0i}$ to the vertices of the form $v_{0i}$ where $i=2n+1$ and $n \in \{0,\cdots,\frac{r-2}{2}\}$ when $d-c>0$.
Finally, a vertex $z$ is added and joined to $w_0$, $u_0$ and $w_1$.
In the case wherein $d-c>0$ is even, we add an additional twin $v'_{1{(r-1)}}$ to the vertex $v_{1{(r-1)}}$. 
We define the set $V$, $Y$ and $X$ as follows:
\begin{center}
    $V=\{v_{1{(r-1)}},\cdots,v_{(c-3){(r-1)}}\}$, \quad
    $Y=\{v_{1r},\cdots,v_{(c-3)r}\}$
    \[
X \;=\; \begin{cases}
\{v_{0i},v'_{0i} : i=2n+1 \text{ where } n \in \{0,\cdots,\frac{r-2}{2}\} \}, &\text{if } d-c \text{ is odd},\\[1mm]
\{v_{0i},v'_{0i} : i=2n+1 \text{ where } n \in \{0,\cdots,\frac{r-2}{2}\} \} \cup \{v'_{1{(r-1)}}\}, &\text{if } d-c>0 \text{ is even},\\[1mm]
\emptyset, &\text{if } d-c=0
\end{cases}
    \]
\end{center}
This is the construction of $G_{2,3,c,d}$ where $c\geq 4$. Now, we calculate the parameter $g$, $eg$, $seg$ and $meg$ of this construction graph as follows:

Every vertex of $G_{2,3,c,d}$ lies on at least one shortest path of either 
$v_{0r}$--$x_{c-3}$ path or $x_{0}$--$v_{(c-3)r}$ path.
Hence, either $\{v_{0r},x_{c-3}\}$ or $\{x_{0}$,$v_{(c-3)r}\}$ is a geodetic set and
\[
g(G_{2,3,c,d}) = 2.
\]
Observe that the edge $z w_0$ does not lie on any shortest either $v_{0r}$--$x_{c-3}$ path or $x_{0}$--$v_{(c-3)r}$ path.
Therefore, it forces the vertex $z$ must belong to every edge-geodetic set to cover all incident edges of $z$.
Moreover, since no pair of vertices covers all edges of the graph, we take the set with at least $3$ vertices. Let
\[
S_{eg}=\{z,x_0,v_{(c-3)r}\}.
\]
We claim that $S_{eg}$ is an edge-geodetic set.
Indeed, the pair $\{x_0,v_{(c-3)r}\}$ covers all edges of the layered structure of the graph,
while the pairs $\{z,x_0\}$ and $\{z,v_{(c-3)r}\}$ cover the edges incident with $z$.
Hence, the edge-geodetic number of the graph is,
\begin{center}
    $eg(G_{2,3,c,d})=|S_{eg}|=3$
\end{center}
Since $z$ belongs to every edge-geodetic set, Lemma~\ref{lem:eg-implies-seg}
implies that $z$ belongs to every strong edge-geodetic set.
Furthermore, the edges of the paths $x_0 w_0 z$, $x_0$--$v_{0r}$ and the twin vertices on the $z$--$v_{0r}$ paths 
force the inclusion of the vertices $x_0$ and $v_{0r}$ in the strong
edge-geodetic set.

For each $i \in \{1,\dots,c-3\}$, the edge $v_{i{(r-1)}} v_{ir}$ lies on a unique
shortest $v_{i{(r-1)}}$--$v_{ir}$ path. Hence, at least one of $v_{i{(r-1)}}$ or $v_{ir}$
must belong to any strong edge-geodetic set when $d-c$ is odd. When $d-c$ is even, $v_{ir}$
must belong to any strong edge-geodetic set to cover the $v'_{1(r-1)}v_{1r}$.
Consequently, every strong edge-geodetic set has cardinality at least $c$.
On the other hand, the set
$\{x_0,z,v_{0r}\} \cup \{v_{1r},\dots,v_{(c-3)r}\}$ admits a choice of shortest paths whose union covers all edges of $G_{2,3,c,d}$.
Therefore,
\[
seg(G_{2,3,c,d}) = |\{x_0,z,v_{0r}\}|+|Y|=3+c-3=c.
\]
By Lemma~\ref{lem:eg-implies-meg}, the vertex $z$ belongs to every monitoring
edge-geodetic set.
Moreover, for each $i \in \{1,\dots,c-3\}$, at least one of $v_{i{(r-1)}}$ or $v_{ir}$
must belong to every monitoring edge-geodetic set.
However, choosing $v_{ir}$ does not allow all edges to be monitored,
whereas the vertices $v_{i{(r-1)}}$, together with $x_0$, $z$ and vertices of $X$ (as they are twins and by the \cref{lem:twins}),
monitor all edges of the graph.
Thus, the set
\(
\{x_0,z\} \cup \{v_{1(r-1)},\dots,v_{(c-3)(r-1)}\} \cup X
\)
is a monitoring edge-geodetic set, and no smaller such set exists.
Hence,
\[
meg(G_{2,3,c,d}) = \begin{cases}
|\{x_0,z\}|+|V|+|X| = 2 + c-3 + d-c+1 = d, &\text{if } d-c>0,\\[1mm]
|\{x_0,z,v_{00}\}|+|V|+|X|=3+c-3+0=c=d, &\text{if } d-c=0\\[1mm]
\end{cases} 
\]
This completes the proof.
\end{proof}

\subsection{General realization for \texorpdfstring{$g(G)=2$ and $eg(G)\neq3$}{g(G)=2 and eg(G)=3}}
\label{subsec:g2-general}

The previous \Cref{eliminating} and \cref{subsec:23-realization} were devoted to the detailed analysis of the case
$g(G)=2$ and $eg(G)=3$, including both non-existence results and explicit
constructions realizing all admissible parameter quadruples.
We now move beyond this exceptional case and establish a general existence
theorem for all remaining admissible quadruples with $g(G)=2$.

\begin{theorem}\label{thm:2b-neq3cd}
For any positive integers $2=a \leq b \leq c \leq d$, except for $b=3$,
there exists a connected graph $G_{2,b,c,d}$ with
\[
g(G_{2,b,c,d}) = 2,\quad
eg(G_{2,b,c,d}) = b,\quad
seg(G_{2,b,c,d}) = c,\quad
meg(G_{2,b,c,d}) = d.
\]
\end{theorem}

\begin{proof}

\begin{figure}
	\centering	

\begin{subfigure}[b]{0.3\textwidth}
		\centering
		\scalebox{0.85}{\begin{tikzpicture}[
  scale=0.6, rotate=270,
  vertex/.style={circle,draw,fill=black,inner sep=1pt,minimum size=1.8mm},
  small/.style={circle,draw,fill=black,inner sep=1pt,minimum size=1.8mm},
  every edge/.style={thick}
  ]

\node[vertex,label=above:$y$] (y) at (-3,0) {};
\node[vertex,label=below:$x_0$] (x) at (4,0) {};

\node[vertex,label=right:$z_1$] (z1) at (-1,2.5) {};
\node[vertex,label=right:$w_1$] (w1) at (2,2.5) {};
\node[vertex,label=below right:$z_2$] (z2) at (-1,1.5) {};
\node[vertex,label=below right:$w_2$] (w2) at (2,1.5) {};
\node[vertex,label=below right:$z_{b-2}$] (zb-2) at (-1,0) {};
\node[vertex,label=below right:$w_{b-2}$] (wb-2) at (2,0) {};

\node[vertex,label=right:$f_1$] (f1) at (-1,-1.5) {};
\node[vertex,label=right:$v_1$] (v1) at (2,-1.5) {};
\node[vertex,label=below left:$f_2$] (f2) at (-1,-2.5) {};
\node[vertex,label=above left:$v_2$] (v2) at (2,-2.5) {};
\node[vertex,label=left:$f_{c-b+1}$] (fc-b+1) at (-1,-3.5) {};
\node[vertex,label=left:$v_{c-b+1}$] (vc-b+1) at (2,-3.5) {};

\node (extra) at (8,-1) {};

\draw (y)--(z1)--(w1)--(x);
\draw (y)--(z2)--(w2)--(x);
\draw (y)--(zb-2)--(wb-2)--(x);
\draw (y)--(f1)--(v1)--(x);
\draw (y)--(f2)--(v2)--(x);
\draw (y)--(fc-b+1)--(vc-b+1)--(x);

\draw[bend left=60] (wb-2) to (w1);
\draw[bend right=60] (w2) to (wb-2);
\draw[bend right=60] (w1) to (w2);

\draw [dotted, thick] (w2)--(wb-2);
\draw [dotted, thick] (z2)--(zb-2);
\draw [dotted, thick] (v2)--(vc-b+1);
\draw [dotted, thick] (f2)--(fc-b+1);
\end{tikzpicture}}
		\caption{$2,b,c,d$ where $b\neq3$ and $d-c=0$.}
	\end{subfigure}\hfil
	\begin{subfigure}[b]{0.35\textwidth}
		\centering
		\scalebox{0.85}{\begin{tikzpicture}[
  scale=0.6, rotate=270,
  vertex/.style={circle,draw,fill=black,inner sep=1pt,minimum size=1.8mm},
  small/.style={circle,draw,fill=black,inner sep=1pt,minimum size=1.8mm},
  every edge/.style={thick}
  ]

\node[vertex,label=above:$y$] (y) at (-3,0) {};
\node[vertex,label=left:$x_0$] (x) at (4,0) {};

\node[vertex,label=right:$z_1$] (z1) at (-1,2.5) {};
\node[vertex,label=right:$w_1$] (w1) at (2,2.5) {};
\node[vertex,label=below right:$z_2$] (z2) at (-1,1.5) {};
\node[vertex,label=below right:$w_2$] (w2) at (2,1.5) {};
\node[vertex,label=below right:$z_{b-2}$] (zb-2) at (-1,0) {};
\node[vertex,label=below right:$w_{b-2}$] (wb-2) at (2,0) {};

\node[vertex,label=right:$f_1$] (f1) at (-1,-1.5) {};
\node[vertex,label=right:$v_1$] (v1) at (2,-1.5) {};
\node[vertex,label=below left:$f_2$] (f2) at (-1,-2.5) {};
\node[vertex,label=above left:$v_2$] (v2) at (2,-2.5) {};
\node[vertex,label=left:$f_{c-b+1}$] (fc-b+1) at (-1,-3.5) {};
\node[vertex,label=left:$v_{c-b+1}$] (vc-b+1) at (2,-3.5) {};

\draw (y)--(z1)--(w1)--(x);
\draw (y)--(z2)--(w2)--(x);
\draw (y)--(zb-2)--(wb-2)--(x);
\draw (y)--(f1)--(v1)--(x);
\draw (y)--(f2)--(v2)--(x);
\draw (y)--(fc-b+1)--(vc-b+1)--(x);

\draw[bend left=60] (wb-2) to (w1);
\draw[bend right=60] (w2) to (wb-2);
\draw[bend right=60] (w1) to (w2);

\draw [dotted, thick] (w2)--(wb-2);
\draw [dotted, thick] (z2)--(zb-2);
\draw [dotted, thick] (v2)--(vc-b+1);
\draw [dotted, thick] (f2)--(fc-b+1);

\node[vertex,label=left:$x_1$] (x1) at (5,0) {};
\node[vertex,label=left:$x_2$] (x2) at (6,0) {};
\node[vertex,label=left:$x_3$] (x3) at (7,0) {};
\node[vertex,label=left:$x_4$] (x4) at (8,0) {};
\node[vertex,label=left:$x_5$] (x5) at (9,0) {};
\node[vertex,label=left:$x_{r-2}$] (xr-2) at (10.5,0) {};
\node[vertex,label=left:$x_{r-1}$] (xr-1) at (11.5,0) {};
\node[vertex,label=left:$x_r$] (xr) at (12.5,0) {};
\node[vertex,label=right:$x'_2$] (x'2) at (6,1) {};
\node[vertex,label=right:$x'_4$] (x'4) at (8,1) {};
\node[vertex,label=right:$x'_{r-1}$] (x'r-1) at (11.5,1) {};

\node (extra) at (13.8,-1) {};

\draw (x)--(x1)--(x2)--(x3)--(x4)--(x5);
\draw [dotted, thick] (x5)--(xr-2);
\draw (xr-2)--(xr-1)--(xr);

\draw (x1)--(x'2)--(x3);
\draw (x3)--(x'4)--(x5);
\draw (xr-2)--(x'r-1)--(xr);

\end{tikzpicture}
			
		}
		\caption{$2,b,c,d$ where $b\neq3$ and $d-c$ is odd.}
	\end{subfigure}\hfil
	\begin{subfigure}[b]{0.32\textwidth}
		\centering
		\scalebox{0.85}{\begin{tikzpicture}[
  scale=0.6,rotate=270,
  vertex/.style={circle,draw,fill=black,inner sep=1pt,minimum size=1.8mm},
  small/.style={circle,draw,fill=black,inner sep=1pt,minimum size=1.8mm},
  every edge/.style={thick}
  ]

\node[vertex,label=above:$y$] (y) at (-3,0) {};
\node[vertex,label=left:$x_0$] (x) at (4,0) {};

\node[vertex,label=right:$z_1$] (z1) at (-1,2.5) {};
\node[vertex,label=right:$w_1$] (w1) at (2,2.5) {};
\node[vertex,label=below right:$z_2$] (z2) at (-1,1.5) {};
\node[vertex,label=below right:$w_2$] (w2) at (2,1.5) {};
\node[vertex,label=below right:$z_{b-2}$] (zb-2) at (-1,0) {};
\node[vertex,label=below right:$w_{b-2}$] (wb-2) at (2,0) {};

\node[vertex,label=below right:$f_1$] (f1) at (-1,-1.5) {};
\node[vertex,label=right:$v_1$] (v1) at (2,-1.5) {};
\node[vertex,label=below left:$f_2$] (f2) at (-1,-2.5) {};
\node[vertex,label=above left:$v_2$] (v2) at (2,-2.5) {};
\node[vertex,label=left:$f_{c-b+1}$] (fc-b+1) at (-1,-3.5) {};
\node[vertex,label=left:$v_{c-b+1}$] (vc-b+1) at (2,-3.5) {};

\draw (y)--(z1)--(w1)--(x);
\draw (y)--(z2)--(w2)--(x);
\draw (y)--(zb-2)--(wb-2)--(x);
\draw (y)--(f1)--(v1)--(x);
\draw (y)--(f2)--(v2)--(x);
\draw (y)--(fc-b+1)--(vc-b+1)--(x);

\draw[bend left=60] (wb-2) to (w1);
\draw[bend right=60] (w2) to (wb-2);
\draw[bend right=60] (w1) to (w2);

\draw [dotted, thick] (w2)--(wb-2);
\draw [dotted, thick] (z2)--(zb-2);
\draw [dotted, thick] (v2)--(vc-b+1);
\draw [dotted, thick] (f2)--(fc-b+1);

\node[vertex,label=left:$x_1$] (x1) at (5,0) {};
\node[vertex,label=left:$x_2$] (x2) at (6,0) {};
\node[vertex,label=left:$x_3$] (x3) at (7,0) {};
\node[vertex,label=left:$x_4$] (x4) at (8,0) {};
\node[vertex,label=left:$x_5$] (x5) at (9,0) {};
\node[vertex,label=left:$x_{r-3}$] (xr-3) at (10.5,0) {};
\node[vertex,label=left:$x_{r-2}$] (xr-2) at (11.5,0) {};
\node[vertex,label=left:$x_{r-1}$] (xr-1) at (12.5,0) {};
\node[vertex,label=left:$x_r$] (xr) at (13.5,0) {};
\node[vertex,label=right:$x'_2$] (x'2) at (6,1) {};
\node[vertex,label=right:$x'_4$] (x'4) at (8,1) {};
\node[vertex,label=right:$x'_{r-2}$] (x'r-2) at (11.5,1) {};
\node[vertex,label=right:$x'_{r-1}$] (x'r-1) at (12.5,1) {};

\draw (x)--(x1)--(x2)--(x3)--(x4)--(x5);
\draw [dotted, thick] (x5)--(xr-3);
\draw (xr-3)--(xr-2)--(xr-1)--(xr);

\draw (x1)--(x'2)--(x3);
\draw (x3)--(x'4)--(x5);
\draw (xr-3)--(x'r-2)--(x'r-1)--(xr);

\end{tikzpicture}
		}
		\caption{$2,b,c,d$ where $b\neq3$, $d-c>0$ and $d-c$ is even.}
	\end{subfigure}	
\end{figure}

\begin{figure}
\centering

\end{figure}

We construct $G_{2,b,c,d}$ except for $b=3$ in three cases depending on the parity of $d-c$. Define
\[
r=
\begin{cases}
d-c, & d-c\text{ odd},\\[2mm]
d-c+2, & d-c>0 \text{ even},\\[2mm]
0, & d-c=0.
\end{cases}
\]

 At the beginning, drawing a path $P_{r+4}$ and labeling the vertices in order as $y$, $z_1$, $w_1$, $x_0$, $x_1,\dots,x_r$.

 Then, we add $b-3$ parallel paths $ yz_iw_ix_0$ from $y$ to $x_0$ where $i \in \{2,\dots,b-2\}$.
Let
\[
W=\{w_1,\dots,w_{\,b-2}\},
\]
and make $W$ a clique.

 Next again, adding $c-b+1$ additional parallel paths $yf_iv_ix_0$ from $y$ to $x_0$ where $i \in \{1,2, \dots,c-b+1\}$. For each $i\in\{1,\dots,c-b+1\}$, let
\[
V=\{v_1,\dots,v_{c-b+1}\}.
\]

For the last step of the construction, We distinguish the three cases:

\begin{itemize}
    \item[\textbf{(i)}] If \emph{$d-c>0$ even},
    we add open twins $x'_{2i}$ to each vertex $x_{2i}$, $1\le i\le \lfloor \frac{d-c-1}{2} \rfloor$, and add a parallel path 
$x_{r-3}x'_{r-2}x'_{r-1}x_r$. Define
\[
X=\{x_{2i},x'_{2i}:1\le i\le \big\lfloor \frac{d-c-1}{2} \big\rfloor\}\cup\{x_{r-1},x'_{r-1}\}.
\]

\item[\textbf{(ii)}] If \emph{$d-c$ odd},
add open twins $x'_{2i}$ to each $x_{2i}$ for $1\le i\le \lfloor \frac{d-c}{2} \rfloor$.
Set
\[
X=\{x_{2i},x'_{2i}:1\le i\le \lfloor \frac{d-c}{2} \rfloor\}\cup\{x_r\}.
\]
\item[\textbf{(iii)}] If \emph{$d-c=0$},
set $X=\emptyset$.
\end{itemize}
This completes the construction. Now, we present the calculation of $g$, $eg$, $seg$, and $meg$ is as follows:

 Every vertex of the graph $G_{2,b,c,d}$ is covered in at least one of the shortest paths between vertices $y$ and $x_r$. Hence, the geodetic set is $\{y,x_r\}$ and the geodetic number is

\[
        g(G_{2,2,c,d})=|\{y,x_r\}|=2
\]

Observe that none of the shortest paths between $y$ and $x_0$ covers the edges of the clique $W$. Thus, $w_i\in W$ must be chosen as in our edge-geodetic set. Hence,  edge-geodetic set is $\{y,x_0\}\cup W$ and the edge-geodetic number of the graph is
\[
eg(G_{2,b,c,d})=|\{y,x_0\}|+|W|
=2+(b-2)=b.
\]

For the strong edge-geodetic set, we notice that the path $y z_1 w_1 x_0 x_1 \cdots x_r$ 
in the graph $G_{2,b,c,d}$ is a diametral path. Hence, need to include $y$ and $x_r$ as they are the terminal vertices on the opposite ends of the graph and all the vertices of $W$ to cover the edges of the clique. Hence, we can assign that a shortest $y$--$x_r$ path can cover at most one of the $c-b+1$ parallel paths 
$y f_i v_i x_0$. Without loss of generality, assume it covers $yf_1v_1x_0$. Thus, to cover the remaining $yf_iv_ix_0$ paths, all remaining $v_i$ (for $i\ge 2$) must be included. Hence, the strong edge-geodetic set is $\{y,x_r\}\cup W\cup (V\setminus \{v_1\})$ and strong edge-geodetic number is

   \[
    seg(G_{2,b,c,d})=|\{y,x_r\}|+|W|+(|V|-1)=2+b-2+c-b+1-1=c\]

For the calculation of the monitoring edge-geodetic number, we know that 
all edges of the clique require all $w_i$ in $W$. Then, to monitor all edges that lie on the shortest paths between $y$ and $w_i$, $y$ is required to add to the MEG-set. Moreover, to monitor all parallel edges between $y$ and $x_0$, each $v_i$ is required to monitor those parallel paths. 

Furthermore, we know that
all twin vertices must be included in the MEG-set. In this case, we proceed with three cases separately as per the values of $r$,

\begin{itemize}
    \item[\textbf{(i)}] When \emph{$d-c>0$ even}, we need to take twin vertices along with the vertices $x_{r-1}$ and $x'_{r-1}$ to monitor edges of the $C_6$ formed on the end of the path $x_0x_1\dots x_r$. Hence, the MEG-set of $G_{2,b,c,d}$ is $\{y\}\cup W \cup V \cup  X$ and

        \[
        meg(G_{2,b,c,d})=|\{y\}|+|W|+|V|+|X|=1+b-2+c-b+1+2(\frac{d-c-2}{2})+2=d\]
        
    \item[\textbf{(ii)}] When $d-c$ is odd, we need to take twin vertices along with vertex $x_r$ to monitor all final edges $x'_{r-1}x_r$ and $x_{r-1}x_r$. Hence, the MEG-set of $G_{2,b,c,d}$ is $\{y\}\cup W \cup V \cup  X$ and

        \begin{center}
        $meg(G_{2,b,c,d})=|\{y\}|+|W|+|V|+|X|=1+b-2+c-b+1+2\frac{(d-c-1)}{2}+1=d$
        \end{center}
        
   \item[\textbf{(iii)}] When $d-c=0$, there does not exist any twin vertices and $X=\emptyset$. Hence, $y$ to each $w_i$ and each $w_ix_0$ to each $x_0v_i$ and the edges of the clique in $W$ have been monitored by $\{y\}\cup W\cup V$. Hence, the vertex $x_0$ is not required to take as in the MEG-set. Hence, the monitoring edge-geodetic number of the graph is,

        \begin{center}
        $meg(G_{2,b,c,d})=|\{y\}|+|W|+|V|=1+b-2+c-b+1=c=d$
        \end{center}
\end{itemize}

Thus in all cases, $meg(G_{2,b,c,d})=d$, completing the proof.
\end{proof}

While \Cref{thm:2b-neq3cd} guarantees the existence of graphs realizing all
admissible parameter quadruples with $g(G)=2$ and $eg(G)\neq 3$, its proof relies
on several case-dependent arguments.
We now present a general construction framework that subsumes these cases and
systematically produces graphs with prescribed values of the four parameters.

\subsection{Generalizing the construction to arbitrary parameters}\label{subsec:general-construction}

So far, our analysis has focused on graphs with geodetic number $g(G)=2$.
Using the general construction framework introduced in
\Cref{subsec:general-construction}, we now extend these results to arbitrary
values of the geodetic number.
In particular, we show that every admissible parameter quadruple
$(a,b,c,d)$ with $3 \le a \le b \le c \le d$ can be realized by a connected graph.

\medskip

\begin{theorem}\label{thm:abcd-construction}
For any integers $3 \le a \le b \le c \le d$, there exists a connected graph
$G_{a,b,c,d}$ satisfying
\[
g(G_{a,b,c,d}) = a,\qquad eg(G_{a,b,c,d}) = b,\qquad seg(G_{a,b,c,d}) = c,
\qquad meg(G_{a,b,c,d}) = d.
\]
\end{theorem}

\begin{proof}

\begin{figure}[ht]
\centering
\begin{tikzpicture}[
  scale=0.8,
  vertex/.style={circle,draw,fill=black,inner sep=1pt,minimum size=1.8mm},
  small/.style={circle,draw,fill=black,inner sep=1pt,minimum size=1.8mm},
  every edge/.style={thick}
  ]

\node[vertex,label=above:$y$] (y) at (90:1) {};
\node[vertex,label=right:$w_1$] (w1) at (162:1.2) {};
\node[vertex,label=below left:$z$] (z) at (234:1.2) {};
\node[vertex,label=above left:$v_1$] (v1) at (-54:1.2) {};
\node[vertex,label=above right:$x$] (x) at (18:1.2) {};

\draw (w1)--(y)--(x)--(v1)--(z)--(w1);

\node[vertex,label=below:$w_2$] (w2) at ($(w1)+(-0.9,0)$) {};
\node[vertex,label=below:$w_3$] (w3) at ($(w2)+(-0.9,0)$) {};
\node[vertex,label=above:$w_{b-a+1}$] (w4) at ($(w3)+(-1.2,0)$) {};

\foreach \w in {w1,w2,w3,w4}{
  \draw (\w)--(y);
  \draw (\w)--(z);
}

\draw (w1)--(w2)--(w3);

\node[vertex,label=above left:$w_{b-a+2}$] (w5) at ($(w4)+(-1,0.1)$) {};
\node[vertex,label=above:$u_1$] (u1) at ($(w4)+(1.5,2)$) {};
\node[vertex,label=above:$u_2$] (u2) at ($(u1)+(-0.5,0)$) {};
\node[vertex,label=above:$u_3$] (u3) at ($(u2)+(-0.5,0)$) {};
\node[vertex,label=above:$u_{a-3}$] (u4) at ($(u3)+(-1.5,0)$) {};
\node (u5) at ($(u3)+(-0.3,0)$) {};
\node (u6) at ($(u5)+(-1,0)$) {};
\draw [dotted, thick] (u5)--(u6);

\draw[dashed,bend right=25] (w1) to (w5);
\draw[dashed,bend left=35] (w2) to (w5);
\draw[dashed,bend left=35] (w3) to (w5);
\draw [dashed](w4) to (w5);
\draw[dashed,bend right=25] (y) to (w5);
\draw[dashed,bend left=35] (z) to (w5);

\draw[dashed] (u1) to (w5);
\draw[dashed] (u2) to (w5);
\draw[dashed] (u3) to (w5);
\draw [dashed](u4) to (w5);

\draw[dashed,bend right=25] (w1) to (w4);
\draw[dashed,bend left=35] (w2) to (w4);
\draw[dashed] (w3) to (w4);
\draw[bend left=25] (w3) to (w1);
\node[small,label=right:$v_2$] (v2) at ($(z)+(1.8,-0.2)$) {};
\node[small,label=right:$v_3$] (v3) at ($(z)+(2.5,-0.6)$) {};
\node[small,label=right:$v_4$] (v4) at ($(z)+(3.6,-1.5)$) {};
\node (v5) at ($(z)+(2.7,-0.7)$) {};
\node (v6) at ($(z)+(3.3,-1.3)$) {};

\foreach \v in {v1,v2,v3,v4}{
  \draw (z)--(\v);
  \draw (x)--(\v);
  \draw [dotted, thick] (v5)--(v6);
}

\node[small,label=below:$x_1$] (x1) at ($(x)+(0.8,0)$) {};
\draw (x)--(x1);

\node[small,label=below:$x_2$] (x2) at ($(x1)+(0.8,0)$) {};
\node[small,label=below:$x_3$] (x3) at ($(x2)+(0.8,0)$) {};
\node[small,label=below:$x_4$] (x4) at ($(x3)+(0.8,0)$) {};
\node[small,label=below:$x_5$] (x5) at ($(x4)+(0.8,0)$) {};
\node[small,label=below:$x_{r-3}$] (x6) at ($(x5)+(1.5,0)$) {};
\node[small,label=below:$x_{r-2}$] (x7) at ($(x6)+(0.8,0)$) {};
\node[small,label=below:$x_{r-1}$] (x8) at ($(x7)+(0.8,0)$) {};
\node[small,label=below:$x_{r}$] (x9) at ($(x8)+(0.8,0)$) {};
\node[small,label=above left:$x'_{r-2}$] (x11) at ($(x7)+(0,0.8)$) {};
\node[small,label=above:$x'_{r-1}$] (x12) at ($(x8)+(0,0.8)$) {};

\draw (x)--(x1)--(x2)--(x3)--(x4)--(x5);
\draw [dashed](x5)--(x6);
\draw (x6)--(x7)--(x8)--(x9);
\draw (x9)--(x12)--(x11)--(x6);

\node[small,label=above:$x'_2$] (x2p) at ($(x2)+(0,0.8)$) {};
\node[small,label=above:$x'_4$] (x4p) at ($(x4)+(0,0.8)$) {};
\draw (x1)--(x2p)--(x3);
\draw (x3)--(x4p)--(x5);

\end{tikzpicture}
\caption{$G_{a,b,c,d}$ showing the clique $W$ and extension $V$ when $d-c$ is odd.}
\end{figure}
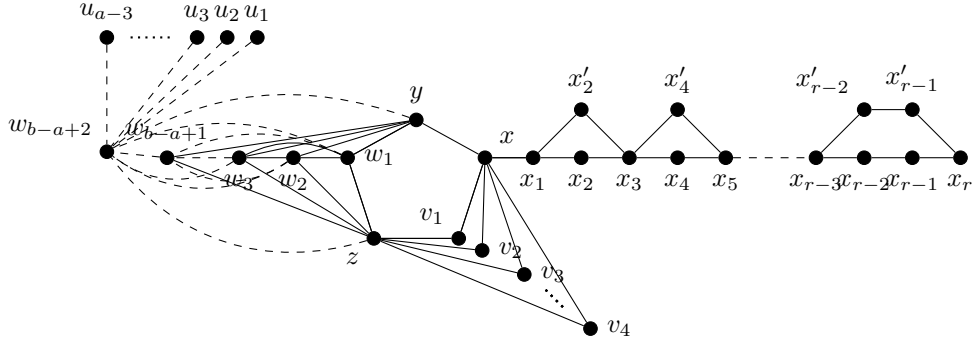


\begin{figure}[ht]
\centering
\begin{tikzpicture}[
  scale=0.8,
  vertex/.style={circle,draw,fill=black,inner sep=1pt,minimum size=1.8mm},
  small/.style={circle,draw,fill=black,inner sep=1pt,minimum size=1.8mm},
  every edge/.style={thick}
  ]

\node[vertex,label=above:$y$] (y) at (90:1) {};
\node[vertex,label=right:$w_1$] (w1) at (162:1.2) {};
\node[vertex,label=below left:$z$] (z) at (234:1.2) {};
\node[vertex,label=above left:$v_1$] (v1) at (-54:1.2) {};
\node[vertex,label=above right:$x$] (x) at (18:1.2) {};

\draw (w1)--(y)--(x)--(v1)--(z)--(w1);

\node[vertex,label=below:$w_2$] (w2) at ($(w1)+(-0.9,0)$) {};
\node[vertex,label=below:$w_3$] (w3) at ($(w2)+(-0.9,0)$) {};
\node[vertex,label=above:$w_{b-a+1}$] (w4) at ($(w3)+(-1.2,0)$) {};

\foreach \w in {w1,w2,w3,w4}{
  \draw (\w)--(y);
  \draw (\w)--(z);
}

\draw (w1)--(w2)--(w3);

\node[vertex,label=above left:$w_{b-a+2}$] (w5) at ($(w4)+(-1,0.1)$) {};
\node[vertex,label=above:$u_1$] (u1) at ($(w4)+(1.5,2)$) {};
\node[vertex,label=above:$u_2$] (u2) at ($(u1)+(-0.5,0)$) {};
\node[vertex,label=above:$u_3$] (u3) at ($(u2)+(-0.5,0)$) {};
\node[vertex,label=above:$u_{a-3}$] (u4) at ($(u3)+(-1.5,0)$) {};
\node (u5) at ($(u3)+(-0.3,0)$) {};
\node (u6) at ($(u5)+(-1,0)$) {};
\draw [dotted, thick] (u5)--(u6);

\draw[dashed,bend right=25] (w1) to (w5);
\draw[dashed,bend left=35] (w2) to (w5);
\draw[dashed,bend left=35] (w3) to (w5);
\draw [dashed](w4) to (w5);
\draw[dashed,bend right=25] (y) to (w5);
\draw[dashed,bend left=35] (z) to (w5);

\draw[dashed] (u1) to (w5);
\draw[dashed] (u2) to (w5);
\draw[dashed] (u3) to (w5);
\draw [dashed](u4) to (w5);

\draw[dashed,bend right=25] (w1) to (w4);
\draw[dashed,bend left=35] (w2) to (w4);
\draw[dashed] (w3) to (w4);
\draw[bend left=25] (w3) to (w1);
\node[small,label=right:$v_2$] (v2) at ($(z)+(1.8,-0.2)$) {};
\node[small,label=right:$v_3$] (v3) at ($(z)+(2.5,-0.6)$) {};
\node[small,label=right:$v_4$] (v4) at ($(z)+(3.6,-1.5)$) {};
\node (v5) at ($(z)+(2.7,-0.7)$) {};
\node (v6) at ($(z)+(3.3,-1.3)$) {};

\foreach \v in {v1,v2,v3,v4}{
  \draw (z)--(\v);
  \draw (x)--(\v);
  \draw [dotted, thick] (v5)--(v6);
}

\node[small,label=below:$x_1$] (x1) at ($(x)+(0.8,0)$) {};
\draw (x)--(x1);

\node[small,label=below:$x_2$] (x2) at ($(x1)+(0.8,0)$) {};
\node[small,label=below:$x_3$] (x3) at ($(x2)+(0.8,0)$) {};
\node[small,label=below:$x_4$] (x4) at ($(x3)+(0.8,0)$) {};
\node[small,label=below:$x_5$] (x5) at ($(x4)+(0.8,0)$) {};
\node[small,label=below:$x_{r-2}$] (x6) at ($(x6)+(0.8,0)$) {};
\node[small,label=below:$x_{r-1}$] (x7) at ($(x7)+(0.8,0)$) {};
\node[small,label=below:$x_{r}$] (x8) at ($(x8)+(0.8,0)$) {};
\node[small,label=above:$x'_{r-1}$] (x9) at ($(x7)+(0,0.8)$) {};

\draw (x)--(x1)--(x2)--(x3)--(x4)--(x5);
\draw [dashed](x5)--(x6);
\draw (x6)--(x7)--(x8);
\draw (x6)--(x9)--(x8);

\node[small,label=above:$x'_2$] (x2p) at ($(x2)+(0,0.8)$) {};
\node[small,label=above:$x'_4$] (x4p) at ($(x4)+(0,0.8)$) {};
\draw (x1)--(x2p)--(x3);
\draw (x3)--(x4p)--(x5);

\end{tikzpicture}
\caption{$G_{a,b,c,d}$ showing the clique $W$ and extension $V$ when $d-c$ is even.}
\end{figure}
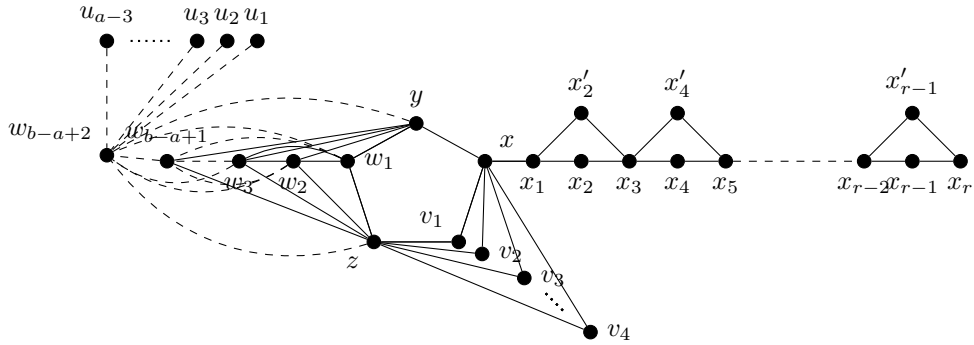

 We present a unified construction of \(G_{a,b,c,d}\) that achieves the desired parameter values.

 Begin the construction with a 5-cycle whose  label of vertices (in counterclockwise order)
\(x,\, y,\, w_1,\, z,\, v_1,\, x.
\)

 Next, we add either $|W|=b-a$ or $|W'|=b-a+1$ new vertices as either
\[
W=\{w_2,w_3,\dots,w_{b-a+1}\},\ or\ \ W'= \{w_2,w_3,\dots,w_{b-a+2}\}
\]
and make either $W\cup\{w_1\}$ or $W'\cup\{w_1\}$ into a clique (if $b=a$ then $W=\varnothing$ and no clique edges are added). Place these clique vertices in the figure between $y$ and $z$ so that each $w_i$ is adjacent to both $y$ and $z$. In this step, if $a=3$, then we add $W$ and otherwise, we add $W'$.

 In the third step of construction, add $|V|=c-b$ vertices
\[
V=\{v_2,\dots,v_{c-b+1}\}.
\]
and for each $v_i\in V$ insert the two edges \(z v_i\) and \(v_i x\).

 In the next step, add $|U|=a-3$ pendent vertices
\[
U=\{u_1,u_2,\dots,u_{a-3}\},
\]
and attach them all to the vertex either \(w_{b-a+1}\) in $W$ or $w_{b-a+2}$ in $W'$ (if $a=3$ then there is no pendent vertex in the construction graph). 

 Now, we add one more pendent vertex $x_r$ attached to the vertex $x$ as an $xx_r$ edge. In the last step of construction, we subdivide the $xx_r$ edges by adding $r-1$ vertices $x_1,x_2,\cdots, x_{r-1}$ where 
\[
r \;=\; \begin{cases}
d-c+3, &\text{if } d-c \text{ is odd},\\[1mm]
d-c+1, &\text{if } d-c \text{ is even},
\end{cases}
\]
and then add the open twin vertices $x'_{2i}$ on each $x_{2i}$ as follows:

 If $(d-c)$ is odd, the twin vertices $x'_{2i}$ for $i=\{1,2,\cdots, \lfloor \frac{(r-3)}{2}\rfloor\} $, let \[
X=\{x_{2i},x'_{2i}: 1\leq i\leq  \lfloor \frac{(r-3)}{2}\rfloor\} \cup \{x_{r-1},x'_{r-1}\},
\]
and add the parallel path from $x_{r-3}$ to $x_r$ with distance $3$, then we can see it as a $6$ cycle form with the labeled $x_{r-3},x_{r-2},x_{r-1},x_r,x'_{r-1},x'_{r-2},x_{r-3}$.  

 If \((d-c)\) is even, then $x'_{2i}$ for $i=\{1,2,\cdots, \lfloor \frac{(r-1)}{2}\rfloor\}$, let \[
X'=\{x_{2i},x'_{2i}: 1\leq i\leq\lfloor \frac{r-1}{2}\rfloor\}\cup \{x_r\},
\]
In this case, we note that if $d-c=0$, then $X'= \emptyset $ and $r=1$.

 This completes the construction of the graph $G_{a,b,c,d}$. 
After establishing this construction completely, we now verify the four parameters one by one.

First, we verify the geodetic number by considering separately the cases \(a=3\) and \(a\ge4\). 

 If \(a=3\), there are no pendent vertices on the \(w_{b-a+1}\). Define
\[
S_g \;:=\; \{y,\,z,\,x_r\}.
\]
We claim that \(S_g\) is a geodetic set of size \(3\).
Since every vertex \(w_i\) is adjacent to both \(y\) and \(z\), hence each \(w_i\) lies on every shortest path between \(y\) and \(z\).
By construction, the shortest paths \(z\!-\!x_r\) and \(y\!-\!x_r\) together cover all vertices and edges along the branches from $x$ to $x_r$.
Moreover, every vertex \(v_j\) with \(1\le j\le c-b+1\) is adjacent to both \(z\) and \(x\), so each \(v_j\) lies on a shortest path between \(z\) and \(x_r\).
Consequently, every vertex of \(G_{3,b,c,d}\) is contained in some shortest path between any two vertices of \(S_g\), implying \(g(G_{3,b,c,d}) \leq 3\).

 Since no nontrivial connected graph satisfies \(g(G)=1\), and any two vertices, such as \(w_i\) and \(x_r\) fail to cover all vertices through their shortest paths in \(G_{3,b,c,d}\), it follows that \(g(G_{3,b,c,d})\ge 3\). Combining this with the upper bound established above, we obtain \(g(G_{3,b,c,d})=3\).

 If \(a \ge 4\), the construction introduces the pendent vertices \(U = \{u_1, u_2, \ldots, u_{a-3}\}\) are attached to \(w_{b-a+2}\).  
Let
\[
S'_g := \{y,\, z,\, x_r\} \cup U.
\]
 Due to Lemma~\ref{simplicial2023}, $U$ must belong to all four parameter sets. Hence, each $u_i$ lies on a unique shortest paths ending at $w_b-a+2$ which itself lies on the shortest path between $y$ and $z$.

 As before, the paths \(y\text{--}x_r\) and \(z\text{--}x_r\) cover all vertices between \(x\) and \(x_r\). Removing any vertex from \(S'_g\) leaves at least one pendent vertex or internal vertex uncovered, so \(S'_g\) is minimal. Hence, \[g(G_{a,b,c,d})= |S'_g| = 3+ a-3=a.\]

 From the preceding arguments, all pendent vertices \(U\) and the terminal vertex \(x_r\) must belong to every edge-geodetic set.
Every shortest path between each pendent vertices and \(x_r\) includes the edges of the paths \(u_i w_{b-a+2} y x \cdots x_r\), thus covering all edges along the branches from \(y\) and \(x\) to \(x_r\).
However, these vertices alone do not cover the edges internal to the clique \(W\) or those forming as \(y w_i z v_j x\) for \(1 \le i \le b-a+1\).  
Therefore, the vertices \(w_i\) (\(1\leq i\leq b-a+1\)) must also belong to every edge-geodetic set.
Additionally, since \(z\) is required to cover the edges incident to each \(v_j\), we must include \(z\) as well.
Consequently, the edge-geodetic set of \(G_{a,b,c,d}\) can be taken as
\[
S_{eg} = U \cup (W \setminus \{w_{b-a+2}\}) \cup \{w_1, z, x_r\},
\]
Hence,
\[
eg(G_{a,b,c,d})=|S_{eg}| = (a-3) + (b-a) + 3 = b.
\]
If \(a=3\), there are no pendent vertices and then $w_i$ with $1\leq i\leq b-a+1$ must belong to the edge-geodetic set. Therefore the edge-geodetic set reduces to
\[
S'_{eg} = (W \cup \{w_1\}) \cup \{z, x_r\},
\]
which again satisfies
\[
eg(G_{3,b,c,d})=|S'_{eg}| = (b-a)+1+2 = b.
\]

 We now determine the strong edge-geodetic number of \(G_{a,b,c,d}\).

 Let \(S_{seg}\) be a strong edge-geodetic set of \(G_{a,b,c,d}\).  
As established in our construction, we note that all vertices of the clique \(W\cup \{w_1\}\) must belong to any strong edge-geodetic set to cover the edges of \(W\cup \{w_1\}\) when \(a=3\). Indeed, to monitor the edges in \(U\cup W'\cup \{w_1\}\) when \(a\geq 4\), \(U\cup (W'\setminus \{w_{b-a+2}\})\cup \{w_1\}\) must belong to any strong edge-geodetic set \(S_{seg}\).
Together with \(U \cup \{x_r\}\), these vertices cover all pendent branches and the edges incident to the clique.  
However, each edges incident to vertices in \(V = \{v_1, \dots, v_{c-b+1}\}\) adjacent to both \(z\) and \(x\) are not yet covered by above vertices even though the vertex $z$ is including in \(S_{seg}\) that can cover one vertex, say $v_1$ between $z$ and $x$. Hence, to ensure coverage, we must include all vertices of \(V\).  
Hence, the strong edge-geodetic set of \(G_{a,b,c,d}\) is:

If $a=3$,
\begin{align*}
    S_{seg} = (W\cup \{w_1\})\cup \{z,x_r\} \cup V,
\end{align*}
and the strong edge-geodetic number is
\[seg(G_{3,b,c,d})=|S_{seg}| = |W|+|\{w_1,z,x_r\}| + |V| = b-a+3 + c-b = c.\]
Otherwise, \[S_{seg} =U\cup (W'\setminus \{w_{b-a+2}\})\cup\{x_r\} \cup V,\]
and the strong edge-geodetic number is
\begin{align*}
seg(G_{a\geq 4,b,c,d})&=|S_{seg}|= |U|+ |(W'\setminus \{w_{b-a+2}\}|+|\{w_1,z,x_r\}| + |V|  \\
&= a-3+b-a+3 + c-b = c.
\end{align*}
Therefore,
\[
seg(G_{a,b,c,d}) = c.
\]

 Finally, we determine the monitoring edge-geodetic number of \(G_{a,b,c,d}\).  
From the construction and by Lemma~\ref{simplicial2023}, all pendent vertices in \(U\) must belong to any MEG-set. Then, to monitor all edges of the clique $W$, the vertices of the clique either $W\cup \{w_1\}$ or $(W'\setminus \{w_{b-a+2}\}) \cup \{w_1\}$ must belong to any MEG-set (Note that $w_{b-a+2}$ is a cut vertex) . Since each vertex $v_i$ of $V\cup \{v_1\}$ is an induced $2$-path $zv_ix$ form which is a part of $C_4$, $v_i$ must belong to MEG-set by Theorem~\ref{4-cycle}.
Moreover, by Lemma~\ref{lem:twins}, every pair of open twins in either \(X\setminus \{x_{r-1},x'_{r-1}\} \) or \(X'\setminus \{x_r\}\) must also be contained in every MEG-set. 
Now, we verify for the remaining edges by considering the following:

If \(d-c\) is even, we have \(r = d-c+1\). Since $X'\setminus \{x_r\}$ belongs to every MEG-set, $x_{r-1}$ and $x'_{r-1}$ are also belong to the every MEG-set of $G_{a,b,c,d}$. However, $x_{r-1}x_r$ and $x'_{r-1}x_r$ edges still remain to be monitored. Since $x'_{r-1}x_rx_{r-1}$ is an induced $2$-path and it is a part of $C_4$, $x_r$ must be inculded in every MEG-set due to Theorem~\ref{4-cycle}. Hence, the MEG-set of $G_{a,b,c,d}$ for $d-c=even$ is either
\[(W\cup \{w_1\})\cup (V\cup \{v_1\})\cup X' \ \ \ \ \ \ or\ \ \ \ \ \ \  U\cup (W\cup \{w_1\})\cup (V\cup \{v_1\})\cup X'  \] 
where \( |X'| = 2\frac{r-1}{2} + 1 = 2\frac{d-c}{2} + 1.\)
\begin{align*}
meg(G_{3,b,c,d})& = |(W\cup \{w_1\})|+(V\cup \{v_1\})|+ |X'|\\
&=(b-3+1)+(c-b+1)+ 2\ \frac{d-c}{2} + 1=d \  \ and\\
    meg(G_{a\geq 4,b,c,d})&=|U|+|(W'\setminus w_{b-a+2})\cup \{w_1\}|+(V\cup \{v_1\})|+ |X'|\\
    &=(a-3)+(b-a+1)+(c-b+1)+2\frac{d-c}{2}+1 =d.
\end{align*}

 If \(d-c\) is odd, we have \(r = d-c+3\). Since $X\setminus \{x_{r-1},x'_{r-1}\}$ belong to every MEG-set, $x_{r-4}$ and $x'_{r-4}$ are also belong to the every MEG-set of $G_{a,b,c,d}$. However, the edges $x_{r-4}x_{r-3}$, $x_{r-3}x'_{r-4}$ and the path $x_{r-3}x_{r-2}x_{r-1}x_rx'_{r-1}x'_{r-2}x_{r-3}$ in $C_6$ form are not monitored yet. As per construction, if we add only one vertex of $C_6$ form, it cannot monitor some edges. Hence, we note that the vertices $x_{r-1}$ and $x'_{r-1}$ must be in the MEG-set of $G_{a,b,c,d}$. Then, all remaining edges will be covered. Therefore, $x_{r-1}$ and $x'_{r-1}$ must be in the MEG-set of  $G_{a,b,c,d}$. Hence, the MEG-set  $G_{a,b,c,d}$ for $d-c=odd$ is either
\[(W\cup \{w_1\})\cup (V\cup \{v_1\})\cup X \ \ \ \ \ \ or\ \ \ \ \ \ \ U\cup (W'\cup \{w_1\})\cup (V\cup \{v_1\})\cup X\]
where \( |X| =2\ \frac{r-4}{2}+ 2 = 2\ \frac{d-c-1}{2} +2.\)
\begin{align*}
meg(G_{3,b,c,d})& = |(W\cup \{w_1\})|+(V\cup \{v_1\})|+ |X|\\
&=(b-3+1)+(c-b+1)+ (d-c+1)=d \  \ and\\
    meg(G_{a\geq 4,b,c,d})&=|U|+|(W'\setminus w_{b-a+2})\cup \{w_1\}|+(V\cup \{v_1\})|+ |X|\\
    &=(a-3)+(b-a+1)+(c-b+1)+(d-c+1) =d.
\end{align*}

\noindent This completes the proof.
\end{proof}


 Having established the existence of graphs realizing arbitrary admissible
parameter quadruples and described an explicit construction for
$G_{a,b,c,d}$, we now briefly analyze the size and computational complexity
of the construction.

\subsubsection{Complexity of the construction}

Let $G_{a,b,c,d}$ be the graph constructed in Theorem~\ref{thm:abcd-construction}. 
The number of vertices and edges in $G_{a,b,c,d}$ grows linearly with $d$.  

Starting from the base graph realizing $g(G)=a$, we introduce three types of gadgets:
\textbf{eg-gadget}, \textbf{seg-gadget}, and \textbf{meg-gadget}, 
each of which contributes a constant number of vertices and edges per increment of the corresponding parameter.  
The necessary gadgets can be attached sequentially to achieve the desired parameters $(b, c, d)$.

Since adding each gadget requires $O(1)$ operations and at most $d$ gadgets are needed, 
the overall construction can be completed in $O(d)$ steps.  
Therefore, the total size of the graph satisfies
\[
|V(G_{a,b,c,d})| = O(d)
\]
and which we summarize in the following corollary.

\begin{corollary}
Given integers $a \le b \le c \le d$, a graph $G_{a,b,c,d}$ realizing
$(g,eg,seg,meg)=(a,b,c,d)$ can be constructed on $O(d)$.
\end{corollary}

\begin{proof}
The construction of $G_{a,b,c,d}$ starts from a fixed base graph realizing $g(G)=a$.
To achieve the target values of $eg(G)=b$, $seg(G)=c$, and $meg(G)=d$, we sequentially
attach a collection of gadgets, each of which increases exactly one of the parameters
by $1$ while preserving the previously achieved values.

Each gadget consists of a constant number of vertices and edges and can be attached
to the existing graph using a constant number of operations.
Since at most $d$ gadgets are required to reach the largest parameter $meg(G)=d$,
the total number of construction steps is $O(d)$
and the entire construction can be performed on $O(d)$.
\end{proof}

\section{Conclusion}\label{sec:conclusion}

In this article, we have conducted a systematic study of connected graphs 
with prescribed values of four related parameters: the geodetic number $g(G)$, 
the edge-geodetic number $eg(G)$, the strong edge-geodetic number $seg(G)$, 
and the monitoring edge-geodetic number $meg(G)$. 

We first identified combinations of parameters that are impossible 
to realize, providing a series of non-existence results for certain 
parameter quadruples. Next, for all admissible quadruples $(a,b,c,d)$, 
we presented explicit constructions of connected graphs $G_{a,b,c,d}$ realizing 
these parameters. Our constructions are flexible and modular, allowing us 
to generate graphs for arbitrary parameter values while maintaining linear 
size and ensuring efficient construction. 

Moreover, we analyzed the computational complexity of our construction, 
demonstrating that $G_{a,b,c,d}$ can be built in $O(d)$ steps, 
with the number of vertices and edges growing linearly with the largest 
parameter $d$. This shows that our approach is not only constructive 
but also practical for generating large graphs with desired properties.

These results complete the characterization of realizable parameter quadruples 
and provide a foundation for further investigations. Future work may include 
studying additional network parameters, exploring extremal properties 
related to graph size and diameter, or extending the analysis to weighted 
or directed graphs.

\medskip

\medskip

\noindent \textbf{Data  availability} This work has no associated data.

\subsection*{Declaration}

\noindent \textbf{Conflict of interest} The authors have no competing interests.

\bibliographystyle{splncs04}
\bibliography{references}

@article{Atici2003,
  title = {On the edge geodetic number of a graph},
  author = {Atici, M.},
  journal = {International Journal of Computer Mathematics},
  volume = {80},
  number = {7},
  pages = {853--861},
  year = {2003}
}

@article{Manuel2017,
  title = {Strong edge geodetic problem in networks},
  author = {Manuel, Paul and Klav{\v{z}}ar, Sandi and Xavier, A. and Arokiaraj, A. and Thomas, E.},
  journal = {Open Mathematics},
  volume = {15},
  number = {1},
  pages = {1225--1235},
  year = {2017}
}

@article{Harary1993,
  title = {The geodetic number of a graph},
  author = {Harary, Frank and Loukakis, E. and Tsouros, C.},
  journal = {Mathematical and Computer Modelling},
  volume = {17},
  number = {11},
  pages = {89--95},
  year = {1993}
}

@inproceedings{foucaud2023monitoring,
  title={Monitoring edge-geodetic sets in graphs},
  author={Foucaud, Florent and Narayanan, Krishna and Ramasubramony Sulochana, Lekshmi},
  booktitle={Conference on Algorithms and Discrete Applied Mathematics},
  pages={245--256},
  year={2023},
  organization={Springer}
}

@article{foucaud2025,
  title={Bounds and extremal graphs for monitoring edge-geodetic sets in graphs},
  author={Foucaud, Florent and Marcille, Clara and Myint, Zin Mar and Sandeep, RB and Sen, Sagnik and Taruni, S},
  journal={Discret. Appl. Math.},
  year={2025}
}

\end{document}